\documentclass[12pt]{article}

\usepackage{amsmath, amssymb, amscd, amsthm}

\usepackage[pdftex, colorlinks=true,  citecolor=blue, linkcolor=blue, linktocpage=true]{hyperref}

\usepackage[all]{xy}

\numberwithin{equation}{section}

\theoremstyle{plain}

\newtheorem{theorem}{Theorem}[section]
\newtheorem{lemma}[theorem]{Lemma}
\newtheorem{proposition}[theorem]{Proposition}
\newtheorem{corollary}[theorem]{Corollary}

\theoremstyle{definition}
\newtheorem{definition}[theorem]{Definition}
\newtheorem{remark}[theorem]{Remark}

\newtheorem{example}[theorem]{Example}

\newcommand\bA{{\mathbb A}}

\newcommand\bG{{\mathbb G}}

\newcommand\bP{{\mathbb P}}

\newcommand\bV{{\mathbb V}}

\newcommand\bZ{{\mathbb Z}}

\newcommand\cI{{\mathcal I}}
\newcommand\cJ{{\mathcal J}}

\newcommand\cO{{\mathcal O}}

\newcommand\cT{{\mathcal T}}

\newcommand\fg{\mathfrak{g}}
\newcommand\fh{\mathfrak{h}}
\newcommand\fm{\mathfrak{m}}

\newcommand\codim{{\rm codim}}

\newcommand\dom{{\rm dom}}

\newcommand\fr{{\rm fr}}

\newcommand\he{{\rm ht}}
\newcommand\id{{\rm id}}

\newcommand\pr{{\rm pr}}

\newcommand\red{{\rm red}}
\newcommand\reg{{\rm reg}}

\newcommand\sing{{\rm sing}}

\newcommand\Aut{{\rm Aut}}
\DeclareMathOperator\Bl{Bl}

\newcommand\Der{{\rm Der}}

\newcommand\GL{{\rm GL}}

\newcommand\Lie{{\rm Lie}}
\newcommand\N{{\rm N}}

\newcommand\Spec{{\rm Spec}}
\DeclareMathOperator\Stab{Stab}

\newcommand\Sym{{\rm Sym}}

\title{Actions of finite group schemes on curves}

\author{Michel Brion}

\date{}

\begin{document}

\maketitle

\begin{abstract}
Every action of a finite group scheme 
$G$ on a variety admits a projective
equivariant model, but not necessarily
a normal one. As a remedy, we introduce 
and explore the notion of $G$-normalization. 
In particular, every curve equipped with 
a $G$-action has a unique projective 
$G$-normal model, characterized by the 
invertibility of ideal sheaves of all orbits.
Also, $G$-normal curves occur naturally
in some questions on surfaces in positive
characteristics.
\end{abstract}

\section{Introduction}
\label{sec:int}

Much is known about finite groups of
automorphisms of algebraic varieties, 
the case of curves being the most 
classical and well-understood. By contrast, 
finite group schemes of automorphisms
(in characteristic $p > 0$) seem to 
have attracted little attention until very 
recent years, where they have been 
determined for several natural classes 
of surfaces; see e.g. \cite{Tziolas}, 
\cite{DD}, \cite{Ma20}, \cite{Ma22}.
One reason may be that key properties
of finite group actions fail for 
a finite group scheme $G$, for example: 

\begin{itemize}

\item 
If $G$ acts faithfully on a variety $X$, 
then it may not act freely on a dense open 
subset.

\item
The $G$-action on $X$ may not lift 
to an action on the normalization.  
 
\end{itemize}

When $X$ is a normal projective curve,
its group of birational automorphisms
coincides with the automorphism group,
and hence is the group of rational
points of an algebraic group.
Moreover, every finite group is the 
full automorphism group scheme of 
some smooth projective curve over 
an algebraically closed field (see 
\cite{MV}). These results also do 
not extend to the schematic setting,
already for infinitesimal group schemes
of height at most $1$ (which are 
in bijective correspondence with 
finite-dimensional $p$-Lie algebras):

\begin{itemize}

\item
Every curve admits $p$-Lie algebras 
of rational vector fields of any 
prescribed dimension.

\item
But there are strong restrictions on 
the $p$-Lie algebras which can be realized 
by rational vector fields on a curve, 
as follows from \cite[Thm.~12.1]{ST}. 

\end{itemize}

In this paper, we propose remedies to
some of these failures. 
As a substitute for ``generic freeness'', 
we show that every action
of a finite group scheme $G$ on a variety
$X$ is ``generically transitive'' 
(Corollary \ref{cor:homogeneous}; 
if $G$ is infinitesimal of height $1$, 
this is \cite[Prop.~5.2]{ST}). 

As a substitute for the normalization, 
we introduce and explore the notion of 
``$G$-normalization'' in Section \ref{sec:Gnor}. 
In particular, we show that every rational 
$G$-action on $X$ admits a projective 
$G$-normal model, which is unique 
if $X$ is a curve (Corollary 
\ref{cor:Gnormodel}). We also obtain a version
of Serre's criterion for $G$-normality
(Theorem \ref{thm:serre}), which takes
a more specific form in dimension $1$:
a curve is $G$-normal if and only if
the ideal sheaf of every $G$-orbit is 
invertible (Corollary \ref{cor:Gnormal}).

Let us emphasize that $G$-normal curves 
are generally singular. On the positive 
side, they turn out to be geometrically
unibranch (Corollary \ref{cor:unibranch})
and local complete intersections 
(Corollary \ref{cor:lci}).
Moreover, the tangent sheaf of every
generically free $G$-normal curve is 
invertible (Proposition \ref{prop:tangent}).

Finally, $G$-normal curves are related to
regular surfaces via the following 
construction: let $X$ be a curve equipped
with a $G$-action, and assume that
$G$ is a subgroup scheme of a smooth
connected algebraic group $G^{\#}$
of dimension $1$. Consider the
diagonal action of $G$ on 
$G^{\#} \times X$; then the quotient 
$S$ is a surface equipped with 
a $G^{\#}$-action. Moreover, $X$ is 
$G$-normal if and only if $S$ is regular
(Proposition \ref{prop:regular}).
The case where $X$ is projective and
$G^{\#}$ is an elliptic curve is of
special interest, since $S$ is projective
(then the above construction goes back 
to \cite{BMIII}).  In this case,
 $G$-normal curves provide the missing
ingredient in the recent classification
of maximal connected algebraic groups
of birational automorphisms of surfaces,
see \cite{Fong} and Proposition 
\ref{prop:reduced}.

This paper is organized as follows. 
Section \ref{sec:prel} begins with 
preliminary results on finite group
scheme actions that we could not find
in the literature; we then show that
such actions are generically transitive.
In Section \ref{sec:rat}, we investigate
rational $G$-actions on a variety $X$;
in particular, such rational actions correspond 
bijectively to actions on the generic 
point. This is then used to construct 
faithful rational actions of 
infinitesimal group schemes of height $1$. 
Section \ref{sec:Gnor} makes the first
steps in the study of $G$-normal varieties, 
with applications to curves. The final 
Section \ref{sec:gen} is devoted to 
generically free actions on curves. 
We obtain in particular a local model 
for such an action at a fixed point 
(Proposition \ref{prop:local}). 

We conclude this introduction with two 
open questions. 
For a curve equipped with a generically
free action of an infinitesimal group scheme
$G$, it is easy to see that the Lie algebra
of $G$ has dimension $1$ (Lemma 
\ref{lem:gen}). Of course, this holds 
if $G$ is a subgroup scheme of 
a smooth connected algebraic group of 
dimension $1$. Are there any further examples? 
In particular, are these group schemes 
commutative?

Also, every curve $X$ as above 
may be viewed as a ramified $G$-cover 
of the quotient $Y = X/G$; the latter is normal
if $X$ is $G$-normal. Can we then determine
$X$ in terms of ramification
data on $Y$, in analogy with the known 
description of abelian covers (see e.g. 
\cite{Pardini}, \cite{AP})?

\section{Preliminaries}
\label{sec:prel}

We fix a ground field $k$ of characteristic
$p \geq 0$, and choose an algebraic closure $\bar{k}$.
Given a field extension $K/k$ and a $k$-scheme $X$, 
we denote by $X_K$ the $K$-scheme 
$X \times_{\Spec(k)} \Spec(K)$, with projection
$\pi : X_K \to X$.

A \emph{variety} $X$ is a separated, geometrically 
integral scheme of finite type over $k$.
A \emph{curve} (resp.~a \emph{surface})
is a variety of dimension $1$ (resp.~$2$).

The field of rational functions on a variety 
$X$ is denoted by $k(X)$. This is a 
\emph{function field in $n$ variables}
where $n = \dim(X)$, i.e., a separable,
finitely generated field extension $K$ 
of $k$, such that $k$ is algebraically closed 
in $K$. Conversely, every function field $K$ 
in $n$ variables is the field of rational 
functions on some $n$-dimensional variety $X$,
a \emph{model} of $K$.

Throughout this paper, we denote by $G$
a finite group scheme, and by $\vert G \vert$
its \emph{order}, i.e., the dimension of the
$k$-vector space 
$\cO(G) = \Gamma(G,\cO_G)$. If $p = 0$ then 
$G$ is \'etale by Cartier's theorem (see 
\cite[II.6.1.1]{DG}). This fails if $p > 0$,
where basic examples of non-\'etale finite
group schemes are $\mu_p$ 
(the multiplicative group scheme of $p$th
roots of unity) and $\alpha_p$ (the kernel of
the $p$th power map in the additive group). 

The connected component of the neutral element 
$e \in G(k)$ is denoted by $G^0$; this is 
an \emph{infinitesimal} group scheme, i.e., 
a finite group scheme having a unique point. 
Also, $G^0$ is a normal subgroup scheme of $G$, 
and $\pi_0(G) = G/G^0$ is \'etale.
If $k$ is perfect, then the reduced subscheme 
$G_{\red}$ is the largest \'etale 
subgroup scheme of $G$, and 
$G = G^0 \rtimes G_{\red}$; in particular,
$G_{\red} 
\stackrel{\sim}{\longrightarrow}
\pi_0(G)$
(see \cite[II.5.1.1, II.5.2.4]{DG}).

Returning to an arbitrary ground field $k$,
we denote by $\Lie(G)$ the Lie algebra
of $G$. Then $\Lie(G^0) = \Lie(G)$;
in particular, $\Lie(G) = 0$ if $G$
is \'etale (e.g., if $p = 0$). If $p >0$
then $\Lie(G)$ has the structure of
a finite-dimensional $p$-Lie algebra, 
also called a restricted Lie algebra;
see \cite[II.7.3.4]{DG}.

Still assuming $p > 0$, we denote by
\[ F_X : X \longrightarrow X^{(p)} \]
the relative Frobenius morphism of 
a scheme $X$, and by
\[ F_X^n : X \longrightarrow X^{(p^n)} 
\]
its $n$th iterate, where $n$ is a positive
integer. Given $G$ as above, each 
$F^n_G$ is a homomorphism of group schemes, 
with kernel the \emph{$n$th Frobenius kernel} 
$G_n$. Moreover, $G$ is infinitesimal 
if and only if $G_n = G$ for $n \gg 0$; 
then the smallest such $n$ is 
the \emph{height} $\he(G)$.
The assignment $G \mapsto \Lie(G)$
yields an equivalence of categories 
between finite group schemes of height 
at most $1$
and finite-dimensional $p$-Lie algebras;
moreover, we have $\Lie(G_1) = \Lie(G)$
(see \cite[II.7.3.5, II.7.4.1]{DG}).

A $G$-\emph{scheme} is a scheme $X$ 
equipped with a $G$-action
\[ \alpha : G \times X \longrightarrow X, 
\quad (g,x) \longmapsto g \cdot x. \]
Note that $\alpha$ is identified with
the projection
$G \times X \to X$ via the automorphism
$(\pr_1,\alpha)$ of $G \times X$. 
In particular, the morphism 
$\alpha$ is finite and locally free.
The $G$-action is said to be \emph{faithful}
if every non-trivial subgroup scheme acts
non-trivially.

A morphism of $G$-schemes $f : X \to Y$ is
\emph{equivariant} if 
$f(g \cdot x) = g \cdot f(x)$
identically on $G \times X$.

Given a $G$-scheme $X$, the \emph{stabilizer}
$\Stab_G$ is the preimage of the diagonal under 
the \emph{graph morphism} 
\[ \gamma : G \times X \longrightarrow 
X \times X, \quad
(g,x) \longmapsto (g \cdot x, x). \]
Via the second projection, $\Stab_G$ is 
a closed subgroup scheme of the $X$-group scheme
$G \times X$; in particular, the projection
$\Stab_G \to X$ is finite. 

We now consider a $G$-scheme $X$ of finite type, 
and a closed subscheme $Y \subset X$. The action 
$\alpha$ restricts to a finite morphism
$\alpha_Y : G \times Y \to X$, with schematic 
image denoted by $G \cdot Y$. We say that
$Y$ is \emph{$G$-stable} if $G \cdot Y = Y$;
equivalently, $\alpha_Y$ factors through $Y$. 
For an arbitrary closed subscheme $Y$, note 
that $G \cdot Y$ is the smallest closed 
$G$-stable subscheme of $X$ containing $Y$. 

In particular, taking for $Y$ a closed point 
$x \in X$, we obtain the \emph{$G$-orbit} 
$G \cdot x$. We say that $x$ is 
\emph{$G$-fixed} if it is $G$-stable 
and the induced action of $G$ on 
$\Spec(\kappa(x))$ is trivial, 
where $\kappa(x)$ denotes the residue 
field at $x$. Equivalently, $x$ lies in
the fixed point subscheme $X^G$
(the largest $G$-stable closed subscheme
of $X$ on which $G$ acts trivially).

In the opposite direction, 
the $G$-action is said to be \emph{free} 
at $x \in X$ if the stabilizer 
$\Stab_G(x)$ is trivial. 
We denote by $X_{\fr}$ the set of free points
of $X$; this is an open $G$-stable subset 
of $X$. For a faithful action of an \'etale 
group $G$, it is easy to see that $X_{\fr}$
is non-empty. But this does not extend to
arbitrary faithful actions, as shown by the 
example of $\alpha_p \times \alpha_p$ acting 
on the affine plane $\bA^2$ via
$(u,v) \cdot (x,y) = (x, y + u + x v)$.

\begin{lemma}\label{lem:quot}
Let $X$ be a $G$-scheme of finite type
such that every $G$-orbit is contained
in an open affine subset. Then there is 
a categorical quotient by $G$,
\[ q : X \longrightarrow Y = X/G, \]
where $Y$ is a scheme of finite type.
Moreover, $q$ is finite and surjective,
with fibers at closed points being 
the $G$-orbits (as sets).

If in addition $G$ acts freely on $X$, then
$q$ is faithfully flat and the graph morphism 
$\gamma$ induces an isomorphism 
$G \times X \stackrel{\sim}{\longrightarrow}
X \times_Y X$. 
\end{lemma}

Equivalently, $q$ is a \emph{$G$-torsor} 
if the action is free. 

Lemma \ref{lem:quot} is obtained in 
\cite[\S 12, Thm.~1]{Mumford}
under the assumption that $k$ is algebraically 
closed; the proof extends unchanged to an arbitrary 
field (see \cite[III.2.6.1]{DG} for another proof).

The assumption that every $G$-orbit is 
contained in an open affine subset 
is satisfied if $X$ is quasi-projective 
(then every finite set of points is 
contained in an open affine subset); 
in particular, if $X$ is a curve. 
This assumption is also satisfied
if $G$ is infinitesimal (then the $G$-orbits 
are just fat points);
in that case, the quotient morphism is radicial
and bijective, see e.g. \cite[Lem.~2.5]{Br17a}.
But an example of Hironaka 
(see \cite{Hironaka}) yields an action
of the constant group $\bZ/2\bZ$ on 
a smooth proper threefold which admits no 
categorical quotient.

In the opposite direction, 
if the categorical quotient $q: X \to Y$ 
exists and is finite,
then $X$ is covered by open affine 
$G$-stable subsets. Moreover, for any 
open $G$-stable subset $U$ of $X$, the image
$V = q(U)$ is open in $Y$ and the restriction 
$q\vert_U : U \to V$ is the categorical quotient. 
In particular, we have a $G$-torsor 
$X_{\fr} \to X_{\fr}/G = Y_{\fr}$.

Given a closed $G$-stable subset 
$i : Z \subset X$, the quotient $Z\to Z/G$ 
also exists and hence comes with a morphism 
$i/G: Z/G \to X/G$.
If $G$ is linearly reductive, then $i/G$
is a closed immersion; also, recall that
the linearly reductive groups 
are exactly the extensions of finite 
\'etale groups of order prime to $p$ by groups 
of multiplicative type (see \cite[IV.3.3.6]{DG}).
For an arbitrary group $G$, the morphism
$i/G$ is not necessarily a closed immersion,
as show by the example of $\alpha_p$
acting on $\bA^2$ via 
$u \cdot (x,y) = (x, y + u x)$;
then the quotient is the morphism
$(x,y) \mapsto (x,y^p)$. The zero subscheme
$Z$ of $x$ is a $G$-fixed affine line with 
coordinate $y$, and hence has quotient
the morphism $y \mapsto y$.

Next, recall that the formation of the categorical 
quotient commutes with flat base change on $Y$.
As an easy consequence, for any normal subgroup 
scheme $N \triangleleft G$, we obtain an action of 
$G/N$ on $X/N$ such that the quotient morphism
$X \to X/N$ is equivariant, and the induced morphism 
$(X/N)/(G/N) \to X/G$ is an isomorphism.

\begin{lemma}\label{lem:affine}
Let $X$ be a $G$-scheme of finite type, and 
$U \subset X$ an open subset. Then $U$ contains 
a dense open affine $G$-stable subset.
\end{lemma}

\begin{proof}
The quotient morphism $X \to X/G^0$ exists 
and is finite, radicial and $G$-equivariant, 
where $G$ acts on $X/G^0$ via its \'etale quotient 
$G/G^0 = \pi_0(G)$. 
Moreover, every open subset of $X$ is 
$G^0$-stable. Thus, it suffices to prove the
assertion for the $\pi_0(G)$-scheme $X/G^0$. 

So we may assume that $G$ is \'etale; then $G_{k'}$ 
is constant for some finite Galois field extension 
$k'/k$. We may further assume that $U$ is affine;
then $U_{k'}$ contains 
$\bigcap_{g \in G(k')} g \cdot U_{k'}$
as a dense open affine subset, stable by $G(k')$
and hence by $G_{k'}$, and also by the action of
the Galois group of $k'/k$. The statement follows 
from this by Galois descent. 
\end{proof}

\begin{lemma}\label{lem:field}
Let $X$ be a $G$-variety with function field $K$.
Choose a dense open affine $G$-stable 
subset $U \subset X$. 

\begin{enumerate}

\item[{\rm (i)}]
The field of invariants $L = K^G$ is 
the fraction field of the ring of invariants
$\cO(U)^G$.

\item[{\rm (ii)}] 
The scheme $\Spec(K)$ is the generic fiber 
of the quotient $q : U \to U/G$.  

\item[{\rm (iii)}] 
There is a unique action of $G$ on $\Spec(K)$ 
such that the morphism 
$\Spec(K) \to X$ is equivariant. 

\item[{\rm (iv)}]
The extension $K/L$ is finite. Moreover, 
$K/K^{G^0}$ is purely inseparable, and 
$K^{G^0}/L$ is separable.

\end{enumerate}

\end{lemma}

\begin{proof}
We may replace $X$ with $U$, and hence
assume that $X = \Spec(R)$ where $R$
is a finitely generated algebra equipped
with a $G$-action; moreover, $R$ is a 
domain with fraction field $K$.

(i) Given $l \in L$, the set of those
$r \in R$ such that $r l \in R$ is 
a non-zero $G$-stable ideal $I$ of $R$. 
It suffices to show that $I^G \neq 0$.
For this, we use a norm argument from
\cite[\S 12, p.~112]{Mumford}. 
Observe that 
$\cO(G \times X) = \cO(G) \otimes_k R$
is a finite free $R$-module via the
co-action $\alpha^*$. Denote by
$\N : \cO(G) \otimes_k R \to R$
the corresponding norm map. Then
$\N(\alpha^*(r)) \in R^G$ for all 
$r \in R$, see loc.~cit.
If $r \in I$ then 
$\alpha^*(r) \in \cO(G) \otimes_k I$
as $I$ is $G$-stable. Using the
covariance of the norm
(see e.g. \cite[II.6.5.4]{EGA}), it follows
that $\N(\alpha^*(r)) \in I^G$.
If in addition $r \neq 0$,
then $\alpha^*(r) \neq 0$ and hence
$\N(\alpha^*(r)) \neq 0$. This
completes the proof of~(i).

Next, we prove (ii), (iii) and (iv) 
simultaneously.
Note that $R$ is a $G$-module, 
and hence so is the subalgebra
$L R \subset K$ generated by $L$ and $R$. 
Also, $R$ is a finite module over 
$R^G = \cO(X)^G$, 
and hence $L R$ is a finite-dimensional vector 
space over $L$. Since $L R$ is an integral
domain, it follows that it is a field; thus, 
$L R = K$ as the latter is the fraction field of 
$R$. This yields a $G$-algebra structure on $K$ 
extending that on $R$. Also, $K$ is the localization
of $R$ at $L\setminus \{ 0 \}$, and hence the natural 
map $L \otimes_{R^G} R \to K$ is an isomorphism;
equivalently, $\Spec(K)$ is the generic fiber 
of $q : X \to X/G$. It is also the generic fiber 
of the quotient $X \to X/G^0$. Since the latter
is radicial, the extension $K/K^{G^0}$ is purely 
inseparable. Finally, $K^{G^0}/K^G$ is separable
as $X/G^0 \to X/G$ is the quotient by the finite 
\'etale group $\pi_0(G)$. 
\end{proof}

If $G$ is a constant group scheme acting faithfully
on $X$, then $G = \Aut_L(K)$ is uniquely determined 
by the invariant subfield $L \subset K$. This does 
not extend to actions of (say) infinitesimal group
schemes: for example, the $p$th power map of $\bA^1$ 
is the quotient by the actions of $\mu_p$ via 
$t \cdot x = tx$, and of $\alpha_p$ via 
$u \cdot x = x + u$.

\begin{proposition}\label{prop:generic}
Let $X$ be a $G$-variety. Then there exists
a dense open affine $G$-stable subset 
$U \subset X$ such that the graph morphism
\[ G \times U \longrightarrow U\times_{U/G} U,
\quad (g,x) \longmapsto (x, g \cdot x) \] 
is faithfully flat. 
\end{proposition}

\begin{proof}
This is known in the setting of
actions of smooth algebraic groups,
as a modern version of a result of 
Rosenlicht (see \cite{Springer}).
We will deduce the desired statement 
from this version, after some first 
reductions.

We may assume that $X$ is affine; then there
exists a quotient morphism $q: X \to Y$. 
Using generic flatness (see e.g.
\cite[${\rm IV}_{\rm 2}$.6.9.1]{EGA}),
we may further assume 
that $q$ is faithfully flat.

Next, we may assume that $G$ is a 
subgroup scheme of a smooth connected 
affine algebraic group $G^{\#}$ 
(for example, $G^{\#} = \GL_n$ in which $G$ 
is embedded via the regular representation). 
Then $G^{\#} \times X$ is an affine variety
equipped with a free action of $G$ via 
$g \cdot (g^{\#},x) = (g^{\#} g^{-1}, g \cdot x)$
and the quotient by this action is an affine 
variety $X^{\#} = G^{\#} \times^G X$
on which $G^{\#}$ acts via its action on 
itself by left multiplication. 
(The formation of $X^{\#}$ is functorial
in $X$ once an embedding $G \to G^{\#}$ 
is fixed).
Note that the open $G^{\#}$-stable subsets of 
$X^{\#}$ are exactly the subsets 
$U^{\#} = G^{\#} \times^G U$,
where $U \subset X$ is open and $G$-stable;
moreover, $U^{\#}$ is affine if and only if
$U$ is affine. 

The projection $G^{\#} \times X \to G^{\#}$ 
induces a morphism 
$\psi: X^{\#} \to G^{\#}/G$,
which is $G^{\#}$-equivariant and has 
fiber $X$ at the base point of $G^{\#}/G$. 
Also, the projection 
$G^{\#} \times X \to X$ induces a morphism 
$q^{\#} : X^{\#} \to Y$ 
which is the categorical quotient by $G^{\#}$; 
we have 
$\cO(Y) = \cO(X)^G 
\stackrel{\sim}{\longrightarrow} 
\cO(X^{\#})^{G^{\#}}$
and 
$k(Y) = k(X)^G 
\stackrel{\sim}{\longrightarrow} 
k(X^{\#})^{G^{\#}}$. As a consequence,
the algebra $\cO(X^{\#})^{G^{\#}}$
is finitely generated and its fraction field
is $k(X^{\#})^{G^{\#}}$ 
(Lemma \ref{lem:field}). Moreover, 
$q^{\#}$ is faithfully flat, since so are
the quotient morphism 
$G^{\#} \times X \to X^{\#}$ and the composite
morphism 
$G^{\#} \times X 
\stackrel{\pr}{\longrightarrow} X
\stackrel{q}{\longrightarrow} Y$
(use \cite[${\rm IV}_{\rm 2}$.2.2.11]{EGA}).

In view of \cite[Satz~1.7]{Springer}, it follows
that the fiber product $X^{\#} \times_Y X^{\#}$ 
is a variety, and the graph morphism
\[ \varphi^{\#} : G^{\#} \times X^{\#} 
\longrightarrow 
X^{\#} \times_Y X^{\#}, \quad 
(g,x) \longmapsto (g \cdot x, x)
\]
is dominant.
Also, $\varphi^{\#}$ is equivariant for the
action of $G^{\#} \times G^{\#}$ on 
$G^{\#} \times X^{\#}$ defined by
$(g_1,g_2) \cdot (g,x) = 
(g_1 g g_2^{-1}, g_2 \cdot x)$, 
and its action on $X^{\#} \times_Y X^{\#}$ 
via $(g_1,g_2) \cdot (x_1,x_2) = 
(g_1 \cdot x_1, g_2 \cdot x_2)$.
So the image of $\varphi^{\#}$ contains 
a dense open subset $V$ of 
$X^{\#} \times_Y X^{\#}$,
stable by $G^{\#} \times G^{\#}$. Thus, 
$\varphi^{\#,-1}(V)$ is a dense open subset
of $G^{\#} \times X^{\#}$, stable by 
$G^{\#} \times G^{\#}$
and hence of the form $G^{\#} \times U^{\#}$
where $U^{\#} \subset X^{\#}$ is open and 
$G^{\#}$-stable.
Replacing $X^{\#}$ with $U^{^\#}$, we may thus 
assume that $\varphi^{\#}$ is surjective. 
Likewise, using generic flatness and 
equivariance, we may further assume that 
$\varphi^{\#}$ is flat. 

The $G^{\#}$-equivariant morphism 
$\psi : X^{\#} \to G^{\#}/G$ yields a 
$G^{\#} \times G^{\#}$-equivariant morphism
\[ \psi^{\#} : X^{\#} \times_Y X^{\#} 
\longrightarrow G^{\#}/G \times G^{\#}/G, \]
which is faithfully flat by equivariance.
The composite morphism
\[ \psi^{\#} \circ \varphi^{\#} : 
G^{\#} \times X^{\#} 
\longrightarrow  G^{\#}/G \times G^{\#}/G \]
is faithfully flat as well, and its fiber
at the base point is $G \times X$. 
Moreover, the restriction of the graph morphism
$\varphi^{\#}$ to this fiber is 
the analogously defined morphism 
$\varphi : G \times X \to X \times_Y X$.
By the fiberwise criterion for flatness
(see \cite[${\rm IV}_{\rm 3}$.11.3.11]{EGA}), 
it follows that $\varphi$ is faithfully flat.
\end{proof}

As a direct consequence of the above proposition,
the $G$-action on $X$ is generically transitive.
Indeed, the generic fiber of the structure morphism 
$U \times_{U/G} U \to U/G$ is
\[ Z = \Spec(K \otimes_L K) = 
\Spec(K) \times_{\Spec(L)} \Spec(K) \]
by Lemma \ref{lem:field}. This is a $K$-scheme
via the first projection, and has a canonical
$K$-point $z$ (the diagonal) corresponding to 
the multiplication $K \otimes_L K \to K$.
Moreover, the $G$-action on $\Spec(K)$ 
yields a $G_K$-action on $Z$, 
and Proposition \ref{prop:generic} implies that
the orbit map $G_K \to Z$, $g \mapsto g \cdot z$ 
is faithfully flat. Also, $H = \Stab_{G_K}(z)$
is the generic fiber of the projection 
$\Stab_G \to X$,
i.e., the generic stabilizer. This yields 
the following:

\begin{corollary}\label{cor:homogeneous}
With the above notation, we have
a $G_K$-equivariant isomorphism
$Z \simeq G_K/H$. 
\end{corollary}

Considering the lengths of the finite $K$-schemes
$Z$ and $G_K/H$, it follows that
$[K : L] = [G_K : H]$. This divides 
$\vert G \vert$, and equality holds 
if and only if $H$ is trivial.
Equivalently, $G_K$ acts freely on $Z$, i.e., 
$G$ acts freely on $\Spec(K)$. We have proved:

\begin{corollary}\label{cor:degree}
With the above notation, $[K : L]$ divides 
$\vert G \vert$. Moreover, equality holds if and 
only if the $G$-action on $X$ is generically free.
\end{corollary}

\section{Rational actions}
\label{sec:rat}

\begin{definition}\label{def:rat}
Let $G$ be a finite group scheme, and $X$ 
a variety. A \emph{rational action} of 
$G$ on $X$ is a rational map
$\alpha : G \times X \dasharrow X$
which satisfies the following two properties:

\begin{enumerate}
\item[{\rm (i)}]
The rational map
$(\pr_1,\alpha) : 
G \times X \dasharrow G \times X$
is birational.

\item[{\rm (ii)}]
The rational maps
$\alpha \circ (\id_G \times \alpha), 
\alpha \circ (\mu \times \id_X) :
G \times G \times X \dasharrow X$ 
are equal, where $\mu : G \times G \to G$ 
denotes the multiplication.

\end{enumerate}

\end{definition}

In this definition (adapted from 
\cite[\S 3]{Demazure}),
a \emph{rational map} $f: Y \dasharrow Z$ 
is an equivalence class of pairs $(U,\varphi)$, 
where $U$ is a schematically
dense open subset of $Y$, and $\varphi : U \to Z$
a morphism; two pairs $(U,\varphi)$, $(V,\psi)$
are equivalent if there exists a schematically dense
open subset $W \subset U \cap V$ such that
$\varphi \vert_W = \psi \vert_W$. Every rational map
$f : X \dasharrow Y$ has a unique representative
$(U,\varphi)$, where $U$ is maximal; then $U$
is the \emph{domain of definition} $\dom(f)$. 
The formation of the domain of definition
commutes with base change by field extensions
(see \cite[${\rm IV}_{\rm 4}$.20.3.11]{EGA})). 
The rational map 
$f$ is \emph{birational} if it admits a representative
$(U,\varphi)$ such that $\varphi$ is an isomorphism
onto a schematically dense open subset of $Z$.

Since $X$ is reduced, every dense open subset 
$U \subset X$ is schematically dense. Also, 
we may replace $X$ with any dense open subset
in Definition \ref{def:rat}, and hence assume 
that $X$ is smooth. Then $G \times X$ is
Cohen-Macaulay (as $G$ is finite), and hence 
has no embedded component. As a consequence, 
we may replace ``schematically dense'' with 
``dense'' when dealing with the rational maps 
in (i), (ii).

Let $\alpha : G \times X \dasharrow X$ 
be a rational map satisfying (i). 
Then $\alpha$ is dominant, since 
it is the composition of the birational map 
$(\pr_1,\alpha)$ and the second projection. 
Thus, the image of $\alpha$ contains 
a dense open subset $W \subset X$. 
So the image of $\id_G \times \alpha$
contains $G \times W$. Denote by $V$ the
domain of definition of $\alpha$; then
the composite rational map
$\alpha \circ(\id_G \times \alpha)$ 
is defined on the open subset 
$(\id_G \times \alpha)^{-1}(V \cap (G \times W))$. 
Moreover,  the composite rational map 
$\alpha \circ (\mu \times \id_X)$ is defined 
as $\mu \times \id_X$ is a morphism (see 
\cite[${\rm IV}_{\rm 1}$.20.3.1]{EGA}). 
So the two
compositions of rational maps in (ii) make sense.

For any $g \in G$, the intersection 
$V_g = \pr_1^{-1}(g) \cap V_{\kappa(g)}$ 
is identified with a dense open subset of 
$X_{\kappa(g)}$;
moreover, $g$ induces a birational morphism
$\alpha_g : V_g \to X_{\kappa(g)}$ (as follows from 
the condition (i) and the finiteness of $G$).
This motivates (i), while (ii) is a rational
analogue of the associativity property of
actions.

\begin{proposition}\label{prop:reg}
Let $X$ be a variety equipped with a rational 
action $\alpha$ of $G$. Then there exists 
a dense open subset $U \subset X$ such that
$\alpha$ is defined on $G \times U$ and induces
a $G$-action on $U$.
\end{proposition}

\begin{proof}
We first consider the case where $k$ is
algebraically closed; then we have
$G = G^0 \rtimes G_{\red}$ and $G_{\red}$ 
is the constant group scheme associated with $G(k)$. 
For any $g \in G(k)$ and $x \in X$, we denote
$\alpha(g,x)$ by $g \cdot x$ whenever it is defined,
i.e., $x \in V_g$. By (ii), given $g,h \in G(k)$
and $x \in X$, if $h \cdot x$ and 
$g \cdot (h \cdot x)$
are defined, then $gh \cdot x$ is defined and equals
$g \cdot (h \cdot x)$. It follows easily that 
the birational morphism 
$\alpha_e : V_e \to X$ is just
the inclusion. Also, 
\[ W = \bigcap_{g \in G(k)} V_g \] 
is a dense open subset of $X$, as well as
\[ U = \bigcap_{g\in G(k)} \alpha_g^{-1}(W). \]
We check that $U$ satisfies our assertions.

Since $\alpha_e$ is the inclusion, we have 
$U \subset W$.
Let $g \in G(k)$ and $x \in U$; then 
$g \cdot x$ is defined and lies in $W$. 
We claim that $g \cdot x \in U$. 
Otherwise, there exists
$h \in G(k)$ such that $h \cdot (g \cdot x)$
is defined and does not lie in $W$. Then
$hg \cdot x$ is defined and not in $W$, 
a contradiction. This proves the claim. 

By this claim, we have $G(k) \times U \subset V$.
Since $V$ is open in $G \times X$, it follows
that $G \times U \subset V$. Thus, $\alpha$
yields a morphism 
$\alpha_0 : G \times U \to X$,
which factors through $U$ by the claim again. 
This completes the proof when $k$ is 
algebraically closed. 

Next, we consider an arbitrary field $k$. 
Then $X_{\bar{k}}$ is equipped with a
rational action of $G_{\bar{k}}$. The above
construction yields a dense open subset
$U' \subset X_{\bar{k}}$ on which 
$G_{\bar{k}}$ acts, and which is stable under
all automorphisms of $\bar{k}/k$. Thus,
$U'$ is the preimage of a dense
open subset $U \subset X$ under the projection 
$\pi: X_{\bar{k}} \to X$. One may readily
check that $U$ satisfies our assertions,
by using the fact that the formation
of the domain of definition commutes with
field extensions.
\end{proof}

\begin{corollary}\label{cor:model}
Let $X$ be a variety equipped with 
a rational action of $G$. Then $X$ is
equivariantly birationally isomorphic 
to a projective $G$-variety. 
\end{corollary}

\begin{proof}
Using Proposition \ref{prop:reg}, we may assume
that $X$ is a $G$-variety. In view of Lemma 
\ref{lem:affine}, we may further assume 
that $X$ is affine. Then the algebra 
$\cO(X)$ is generated by a finite-dimensional 
$G$-submodule $V$. This yields a closed 
$G$-equivariant immersion of $X$ in the 
corresponding
affine space $\bV(V) = \Spec (\Sym(V))$, 
and hence in its projective completion 
$\bP(V \oplus k)$ (where $G$ acts via 
its linear representation in $V \oplus k$). 
The schematic image of $X$ in $\bP(V \oplus k)$ 
is the desired projective $G$-variety.
\end{proof}

\begin{corollary}\label{cor:generic}
Every rational action of $G$ on $X$ restricts 
to a $G$-action on the spectrum of the
function field $K = k(X)$. 
Conversely, every $G$-action on $\Spec(K)$ 
extends to a unique rational $G$-action on $X$.
\end{corollary}

\begin{proof}
The first assertion is a direct consequence
of Lemma \ref{lem:field} together with Proposition 
\ref{prop:reg}.

For the converse assertion, consider an action
$\alpha : G \times \Spec(K) \to \Spec(K)$
and the corresponding co-action
\[ 
\alpha^* : K \longrightarrow \cO(G) \otimes K, 
\quad 
x_i \longmapsto \sum_j y_{ij} \otimes z_{ij}
\quad (i = 1, \ldots,m),
\]
where $x_1,\ldots, x_m$ generate the field $K$
over $k$. Let $U$ (resp.~$V$) be a dense open 
affine subset of $X$ on which the $x_i$ 
(resp.~$z_{ij}$) are defined; then $\alpha^*$ 
yields a homomorphism of algebras 
$\cO(U) \to \cO(G) \otimes \cO(V)$, or
equivalently,
a morphism $G \times V \to U$. As $U$ is 
dense in $X$, and $G \times V$ is 
(schematically) dense in $G \times X$ (e.g. 
by \cite[${\rm IV}_{\rm 4}$.20.3.5]{EGA}), 
we get a rational map
$\beta : G \times X \dasharrow X$. 
It satisfies the properties (i) and (ii) in view of 
the ``local determination of morphisms'' (see 
\cite[I.6.5]{EGA}), since these properties hold for 
the action of $G$ on $\Spec(K)$.
\end{proof}

\begin{remark}\label{rem:weil}
More generally, Corollary \ref{cor:model} 
holds in the setting of rational actions 
of algebraic groups, by a refinement 
of Weil's regularization theorem 
(see \cite{Weil}, \cite[Thm.~1]{Rosenlicht} 
for the original version, and 
\cite[Thm.~8]{Br22} for the refinement). 
But Proposition \ref{prop:reg} and 
Corollary \ref{cor:generic} do not extend 
to this setting. 
\end{remark}

We now assume that $p > 0$, and use Corollary
\ref{cor:generic} to construct examples of 
faithful rational actions of infinitesimal
group schemes on any variety $X$.
Recall that the function field $K = k(X)$ 
(a separable, finitely generated extension
of $k$ of transcendence degree $n$)
admits a \emph{$p$-basis of length $n$}, 
i.e., a sequence $(x_1,\ldots,x_n) \in K^n$ 
such that the monomials 
$x_1^{m_1} \ldots x_n^{m_n}$,
where $0 \leq m_1, \ldots, m_n \leq p-1$, form
a basis of $K$ over its subfield $k K^p$
(the composite of $k$ and $K^p$ in $K$; this is 
a function field in $n$ variables as well).
Equivalently, the differentials 
$dx_1,\ldots,dx_n$ form a basis of the
$K$-vector space of K\"ahler differentials
$\Omega^1_{K/k} = \Omega^1_{K/kK^p}$
(see \cite[${\rm IV}_{\rm 1}$.21.4.2, 
${\rm IV}_{\rm 1}$.21.4.5]{EGA}).
We denote the dual basis of the $K$-vector 
space of derivations by 
$D_1,\ldots,D_n \in \Der_k(K) = \Der_{kK^p}(K)$. 
Then the $D_i$ commute pairwise
and satisfy $D_i^p = 0$ for $i = 1, \ldots, n$.

Next, recall the equivalence of categories
between infinitesimal group schemes $G$ 
of height at most $1$ and $p$-Lie algebras 
$\fg = \Lie(G)$ (see \cite[II.7.4.1]{DG}.
Under this equivalence, the $G$-actions
on $\Spec(K)$ correspond to the homomorphisms
of $p$-Lie algebras $\fg \to \Der_k(K)$
in view of \cite[II.7.3.10]{DG}. Also,
recall that the $p$-Lie algebra $k$ with 
trivial $p$th power map (resp.~$p$th power
map $t \mapsto t^p$) corresponds to the
group scheme $\alpha_p$ (resp.~$\mu_p$).

In particular, for any 
$f_1, \ldots,f_m \in kK^p$, the derivations 
$f_1 D_1, \ldots, f_m D_1$ commute pairwise and
satisfy $(f_i D_1)^p = 0$ for all $i$. Choosing
$f_1,\ldots,f_m$ linearly independent over $k$,
this yields:

\begin{lemma}\label{lem:alpha}
Every variety of positive dimension admits
a faithful rational action of $\alpha_p^m$
for any $m \geq 1$.
\end{lemma}

Here $\alpha_p^m = 
\alpha_p \times \ldots \times \alpha_p$ 
($m$ factors).
Likewise, considering the derivations
$x_1 D_1, \ldots, x_n D_n$, we obtain
a faithful action of 
$\mu_p^n$ on $\Spec(K)$, or equivalently,
a faithful $\mu_p^n$-action on 
$K$ by algebra automorphisms (see e.g.
\cite[II.2.1.2]{DG}). The latter action
fixes $kK^p$ pointwise, and satisfies 
\[ (t_1,\ldots,t_n) \cdot (x_1,\ldots,x_n) 
= (t_1 x_1, \ldots, t_n x_n) \]
for all $(t_1,\ldots,t_n) \in \mu_p^n$.
We say that this action is \emph{standard in
the $p$-basis $(x_1,\ldots,x_n)$}.

\begin{lemma}\label{lem:pbasis}
Let $X$ be a variety of dimension $n$.
Then $X$ admits a faithful rational action 
of $\mu_p^n$. Every such action $\alpha$ is 
generically free, and standard in some $p$-basis 
$(x_1,\ldots,x_n)$ of $K = k(X)$. Moreover, 
the $x_i$ are uniquely determined by $\alpha$, 
up to multiplication by non-zero elements 
of $kK^p$. If $\alpha$ 
extends to a faithful rational action 
of an infinitesimal group scheme $H$ of height
$1$ normalizing $\mu_p^n$, then $H = \mu_p^n$. 
\end{lemma}

\begin{proof}
Let $G = \mu_p^n$. The existence of a faithful 
rational $G$-action on $X$ follows from 
Corollary \ref{cor:generic} together with
the preceding construction. 

Given such an action $\alpha$, 
the corresponding $G$-action on $K$ 
yields a grading
\[ 
K = \bigoplus_{\gamma \in \Gamma} K_{\gamma}, 
\]
where $\Gamma = (\bZ/p\bZ)^n$ is the character 
group of the diagonalizable group $G$
(see \cite[II.2.2.5]{DG}). Then 
$K_0 = K^G$ is a subfield of $K$,
and each $K_{\gamma}$ is a $K_0$-vector space.
Moreover, the group $\Gamma$ is generated by 
the $\gamma$ such that $K_{\gamma} \neq 0$,
since the $G$-action is faithful. But these
$\gamma$ form a subgroup of $\Gamma$ as $K$
is a field. It follows that $K_{\gamma} \neq 0$
for all $\gamma \in \Gamma$. Since 
$K_0 \supset kK^p$ and 
$[K:kK^p] = p^n = \vert \Gamma \vert$, 
we must have $K_0= kK^p$ and 
$\dim_{K_0}(K_{\gamma}) = 1$ for all such
$\gamma$. Choosing $x_i \in K_{\gamma_i}$
where $\gamma_1,\ldots,\gamma_n$ form the
standard basis of $(\bZ/p\bZ)^n$, we see that 
the monomials $x_1^{m_1} \ldots x_n^{m_n}$,
where $0 \leq m_1, \ldots, m_n \leq p-1$,
have pairwise distinct weights, and
hence are linearly independent over $kK^p$. 
For dimension reasons, these monomials form 
a $p$-basis. By construction, the action
$\alpha$ is standard in this basis, 
which is unique 
up to non-zero elements of $kK^p$; moreover,
the generic stabilizer is trivial.

Let $H$ be an infinitesimal group scheme 
of height $1$ normalizing $G$, and equipped 
with a faithful rational action on $X$ 
extending $\alpha$. Since the automorphism group 
scheme of $G$ is constant (see 
\cite[III.5.3.3]{DG}), we see that
$H$ centralizes $G$. 
Then $\fg \subset \fh \subset \Der_k(K)$,
where $\fh = \Lie(H)$ centralizes $\fg$.
But we have
\[ \Der_k(K) = \bigoplus_{i=1}^n K D_i 
= \bigoplus kK^p x_1^{m_1} \cdots x_n^{m_n} D_i, 
\]
the latter sum being over $i = 1, \ldots, n$
and $m_1,\ldots,m_n = 0, \ldots, p-1$. Also,
each $x_1^{m_1} \cdots x_n^{m_n} D_i$ is a 
$G$-eigenvector of weight 
$(m_1, \ldots, m_i - 1, \ldots, m_n)$ 
(viewed in $(\bZ/p\bZ)^n$). It follows that 
the centralizer
of $\fg$ in $\Der_k(K)$ is the Lie algebra
\[ k K^p \fg = \big\{ \sum_{i=1}^n t_i x_i D_i
~\vert~ t_1, \ldots, t_n \in k K^p \big\}. \] 
So $\fh$ is a finite-dimensional subspace 
of the $k$-vector space $k K^p \fg$, 
stable under the $p$th power map. If 
$\sum_i t_i x_i D_i \in \fh$, then 
$\sum_i t_i^p x_i D_i, \sum_i t_i^{p^2} x_i D_i,
\ldots \in \fh$. It follows that
$t_i,t_i^p, t_i^{p^2}, \ldots$ are linearly
dependent over $k$ for $i = 1, \ldots, n$. 
In particular, each $t_i$ is algebraic over $k$. 
Since $k$ is algebraically closed in $K$, 
this forces $t_1,\ldots,t_n \in k$
and $\fh = \fg$.
\end{proof}

\begin{remark}\label{rem:infty}
In particular, $X$ admits (many)
faithful rational actions of $\mu_p^n$,
but no faithful rational action of
$\mu_p^{n+1}$. The latter fact also 
follows from a classical result in the
theory of $p$-Lie algebras: choosing
a $p$-basis $(x_1,\ldots,x_n)$ of $K/k$
yields an isomorphism of $kK^p$-algebras
\[ K \simeq 
kK^p[T_1,\ldots,T_n]/
(T_1^p - x_1^p,\ldots,T_n^p - x_n^p) \]
and hence an isomorphism of $K$-algebras
\[ K \otimes_{kK^p} K \simeq 
K[T_1,\ldots,T_n]/(T_1^p,\ldots,T_n^p). \]
As a consequence, the $kK^p$-algebra
$\Der_k(K) = \Der_{kK^p}(K)$ is a form of
the $K$-algebra
$\Der_K(K[T_1,\ldots,T_n]/
(T_1^p,\ldots,T_n^p))$.
The latter is a $p$-Lie algebra over $K$,
known as the split Jacobson--Witt algebra
$W_n$. Its maximal tori (i.e., the 
maximal $p$-Lie subalgebras having a 
basis $D_1,\ldots,D_m$ such that the $D_i$
commute pairwise and satisfy $D_i^p = D_i$)
have been determined in \cite{Demushkin};
in particular, they all have dimension~$n$. 

In another direction, the above construction 
of faithful rational actions of $\mu_p^n$
via $p$-bases can be iterated
to yield faithful rational actions of
$\mu_{p^s}^n$ for any $s \geq 1$.
Indeed, taking $p$th powers in the equality
\[ K = 
\bigoplus_{0 \leq m_1,\ldots, m_n \leq p-1}
k K^p x_1^{m_1} \cdots x_n^{m_n}, \]
we obtain 
\[ k K^p = 
\bigoplus_{0 \leq m_1,\ldots, m_n \leq p-1}
k K^{p^2} x_1^{p m_1} \cdots x_n^{pm_n}, \]
and hence 
\[ K = 
\bigoplus_{0 \leq m_1,\ldots, m_n \leq p^2-1}
k K^{p^2} x_1^{m_1} \cdots x_n^{m_n}. \]
By induction, this yields
\[ K = 
\bigoplus_{0 \leq m_1,\ldots, m_n \leq p^s-1}
k K^{p^s} x_1^{m_1} \cdots x_n^{m_n} \]
for any $s \geq 1$. We may thus define the
desired action of $\mu_{p^s}^n$ by the same 
formula as the standard $\mu_p^n$-action;
its fixed subfield is $k K^{p^s}$.

Since these $\mu_{p^s}^n$-actions are compatible 
with the standard embeddings
$\mu_{p^s}^n \to \mu_{p^{s+1}}^n$, 
where $s \geq 1$,
we get a faithful rational action 
of the ind-group scheme
$\mu_{p^{\infty}}^n = \varinjlim_s \mu_{p^s}^n$
on $X$.
Note that the family $(\mu_{p^s}, s \geq 1)$  
is schematically dense in the multiplicative 
group $\bG_m$, but $X$ may admit no faithful 
rational $\bG_m$-action (for example 
if $X$ is not geometrically uniruled).

Likewise, $X$ admits a faithful rational action
of the ind-group scheme 
$\alpha_{p^{\infty}}^r = 
\varinjlim_s \alpha_{p^s}^r$ for any $r \geq 1$. 
Also, the $\alpha_{p^s}$ form a schematically 
dense family in the additive group $\bG_a$, 
but $X$ may admit no faithful rational 
$\bG_a$-action. 
\end{remark}

\section{$G$-normality}
\label{sec:Gnor}

Recall that $G$ denotes a finite group scheme.

\begin{definition}\label{def:Gnormal}
A $G$-variety $X$ is $G$-\emph{normal} 
if every finite birational morphism 
of $G$-varieties $f: Y \to X$ is
an isomorphism.
\end{definition}

\begin{proposition}\label{prop:Gnormal}
Let $X$ be a $G$-variety.

\begin{enumerate}

\item[{\rm (i)}]
There exists a finite birational morphism of 
$G$-varieties $\varphi : X' \to X$, 
where $X'$ is $G$-normal.

\item[{\rm (ii)}]
For any morphism $\varphi$ as in {\rm (i)} 
and any finite birational morphism 
of $G$-varieties $f : Z \to X$, 
there exists a unique morphism of $G$-varieties 
$\psi : X' \to Z$ such that 
$\varphi = f \circ \psi$.

\end{enumerate}

\end{proposition}

\begin{proof}
(i) Let $f : Y \to X$ be a finite birational
morphism of $G$-varieties.
Then the normalization morphism
$\eta = \eta_X : \tilde{X} \to X$ 
factors uniquely through 
the analogous morphism 
$\eta_Y : \tilde{X} \to Y$.
Thus, 
$\cO_X \subset f_*(\cO_Y) \subset 
(\eta_X)_*(\cO_{\tilde{X}})$. 
Since $\eta_X$ is finite, we may choose $f$ 
so that the subsheaf 
$f_*(\cO_Y) \subset (\eta_X)_*(\cO_{\tilde{X}})$ 
is maximal among the direct images of structure
sheaves of $G$-varieties equipped with a finite 
birational morphism to $X$.

We claim that $Y$ is $G$-normal
(and hence $f: Y \to X$ is the desired morphism). 
To check this, consider
a finite birational morphism of $G$-varieties
$g : Z \to Y$. Then again, $f \circ g : Z \to X$
factors through $\eta_Z$, and hence we have
$f_*(\cO_Y) \subset (f \circ g)_*(\cO_Z)
\subset (\eta_X)_*(\cO_{\tilde{X}})$.
By maximality, we obtain 
$f_*(\cO_Y) = (f \circ g)_*(\cO_Z)$
and hence $\cO_Y = g_*(\cO_Z)$ as $f$ is finite
surjective.
It follows that $g$ is an isomorphism,
proving the claim.

(ii) There exists a dense open $G$-stable subset
$U \subset X$ such that the induced morphisms
$f^{-1}(U) \to U$ and $\varphi^{-1}(U) \to U$ are 
isomorphisms. Thus, we may identify $U$ with an open 
subset of the fiber product $Z \times_X X'$, stable 
under the natural $G$-action. Let $Y$ be the
schematic closure of $U$ in $Z \times_X X'$;
then $Y$ is a $G$-variety equipped with 
finite birational $G$-morphisms $\psi : Y \to Z$,
$g: Y \to X'$
such that the square
\[ 
\xymatrix{Y \ar[r]^-{\psi} \ar[d]_g 
& Z \ar[d]^f \\
X' \ar[r]^-{\varphi} & X
}\]
commutes. Since $X'$ is $G$-normal, 
$g$ is an isomorphism; this yields the desired
morphism $X' \to Z$.
\end{proof}

With the above notation, we say that
$\varphi : X' \to X$ is 
the $G$-\emph{normalization}; it is unique 
up to unique $G$-isomorphism.

\begin{remark}\label{rem:normal}
If a $G$-variety $X$ is $H$-normal for some
subgroup scheme $H$ of $G$, then clearly
$X$ is $G$-normal. In particular, every
normal $G$-variety is $G$-normal. The converse
holds if $G$ is \'etale, since the $G$-action
lifts uniquely to an action on the normalization
(see e.g. \cite[Prop.~2.5.1]{Br17b}).
So the notion of $G$-normalization is only 
relevant in characteristic $p > 0$.
\end{remark}

\begin{corollary}\label{cor:Gnormodel}
Let $X$ be a variety equipped with a rational 
action of $G$. Then $X$ is equivariantly
birationally isomorphic to a $G$-normal
projective variety $Y$. If $X$ is a curve, 
then $Y$ is unique.
\end{corollary}

\begin{proof}
The first assertion follows readily from
Corollary \ref{cor:model} together with
the existence of the $G$-normalization.

Assume that $X$ is a $G$-curve and consider 
two projective $G$-models $Y_1,Y_2$; then 
we have a $G$-equivariant rational map
$f : Y_1 \dasharrow Y_2$. Using Lemma
\ref{lem:affine}, we may find dense open
$G$-stable subsets $U_i \subset Y_i$ ($i = 1,2$)
such that $f$ restricts to an isomorphism
$U_1 \stackrel{\sim}{\longrightarrow} U_2$.
By a graph argument as in the proof 
of Proposition \ref{prop:Gnormal} (ii), 
this yields a projective $G$-curve 
$Y$ equipped with equivariant birational 
morphisms to $Y_1,Y_2$. The second assertion 
follows from this, as every birational 
morphism of projective curves is finite.
\end{proof}

\begin{lemma}\label{lem:quotient}
Let $X$ be a $G$-variety, and $N \lhd G$
a normal connected subgroup scheme. 

\begin{enumerate}

\item[{\rm (i)}] If $X$ is $G$-normal,
then $X/N$ is $G/N$-normal.

\item[{\rm (ii)}] If $X/N$ is $G/N$-normal
and $N$ acts freely on $X$, then $X$ is 
$G$-normal.

\end{enumerate}
 
\end{lemma}

\begin{proof}
(i) Let $f : Y \to X/N$ be a finite birational
morphism of $G/N$-varieties. 
Arguing again as in the proof 
of Proposition \ref{prop:Gnormal} (ii), we obtain
a $G$-variety $Z$ equipped with finite birational
morphisms $\varphi : Z \to X$,
$\psi: Z \to Y$ such that the square
\[ 
\xymatrix{Z \ar[r]^{\psi}  \ar[d]_{\varphi} 
& Y \ar[d]^f \\
X \ar[r]^-q & X/N
}\]
commutes. Since $X$ is $G$-normal, 
$\varphi$ is an isomorphism. 
Then the resulting morphism 
$\psi \circ \varphi^{-1} : X \to Y$ 
is $N$-invariant over 
a dense open subset of $Y$, and hence 
everywhere. Since $q$ is a categorical
quotient, it follows that $f$ has a section.
Since $f$ is a finite morphism of varieties,
it is an isomorphism.

(ii) We argue similarly, and consider 
a finite birational morphism of $G$-varieties
$f : Y \to X$. Since $N$ acts freely on $X$,
it also acts freely on $Y$. Thus, the 
quotient morphisms 
$q_X : X \to X/N$, $q_Y : Y \to Y/N$ 
are $N$-torsors, and fit in a cartesian square
\[ 
\xymatrix{Y \ar[r]^{q_Y}  \ar[d]_f
& Y/N \ar[d]^{g} \\
X \ar[r]^{q_X} & X/N.
}\]
By fppf descent, it follows that $g$ is finite. 
Also, $g$ is birational as $f$ restricts to 
an isomorphism over a dense open $N$-stable
subset of $X$. Thus, $g$ is an isomorphism,
and hence so is $f$. 
\end{proof}

\begin{corollary}\label{cor:unibranch}
Let $X$ be a $G$-normal variety. Then $X/G^0$
is normal. Moreover, the normalization
$\eta : \tilde{X} \to X$ is radicial and bijective. 
\end{corollary}

\begin{proof}
The variety $X/G^0$ is $\pi_0(G)$-normal 
by Lemma \ref{lem:quotient}; this yields 
the first assertion in view of Remark
\ref{rem:normal}.

Recall that the quotient morphism 
$q : X \to X/G^0$ is radicial and bijective. 
Also, $q \circ \eta : \tilde{X} \to X/G^0$ is 
a finite morphism of normal varieties such that 
the corresponding extension of function fields 
is purely inseparable. 
As a consequence, $q \circ \eta$ 
is radicial and bijective as well (see e.g. 
\cite[II.4.3.8]{EGA}). This implies the 
second assertion by using \cite[I.3.5.6]{EGA}.
\end{proof}

Next, we relate the $G$-normalization
of a $G$-variety $X$ with the (usual) 
normalization of a variety obtained 
from $X$ by ``induction'', as in 
the proof of Proposition \ref{prop:generic}.
More specifically, embed $G$ as a closed 
subgroup scheme of a smooth connected
algebraic group $G^{\#}$. Then 
$G^{\#} \times X$ is a variety equipped with
a $G^{\#} \times G$-action via 
$(a,g) \cdot (b,x) = (ab g^{-1}, g \cdot x)$.
Since $G^0$ is infinitesimal, the quotient variety
$X^{\#} = G^{\#} \times^{G^0} X$ exists; it is 
equipped with an action of 
$G^{\#} \times (G/G^0) = G^{\#} \times \pi_0(G)$ 
together with 
a $G^{\#} \times (G/G^0)$-equivariant morphism
\[ \psi : X^{\#} \longrightarrow G^{\#}/G^0. \] 
Here $G^{\#}/G^0$
is identified with the homogeneous space
$(G^{\#} \times (G/G^0))/G$, where $G$ is
embedded diagonally in $G^{\#} \times (G/G^0)$.
The fiber of $\psi$ at the base point 
of this homogeneous space is $G$-equivariantly
isomorphic to $X$.

Now coonsider the normalization
$\eta^{\#} : \tilde{X}^{\#} \to X^{\#}$.
Since $G^{\#} \times (G/G^0)$ is smooth, 
its action on $X^{\#}$ lifts uniquely 
to an action on $\tilde{X}^{\#}$ such that 
$\eta^{\#}$ is equivariant (see e.g.
\cite[Prop.~2.5.1]{Br17b}). Thus,
$\psi \circ \eta^{\#} : 
\tilde{X}^{\#} \to G^{\#}/G^0$ is
$G^{\#} \times (G/G^0)$-equivariant as well.
So its fiber at the base point is a 
$G$-scheme $Y$ equipped with
a $G$-equivariant morphism 
\[ \mu : Y \longrightarrow X \]
Moreover, the morphism
$G^{\#} \times Y \to \tilde{X}^{\#}$,
$(a,y) \mapsto a \cdot y$ factors 
uniquely through an isomorphism
$G^{\#} \times^{G^0} Y
\stackrel{\sim}{\longrightarrow}
\tilde{X}^{\#}$.

\begin{lemma}\label{lem:induced}
With the above notation,
$\mu$ is the $G$-normalization, and 
the $G^0$-norma\-lization as well.
\end{lemma}

\begin{proof}
By construction, we have a cartesian square
\[ 
\xymatrix{G^{\#} \times Y 
\ar[r]^{\id \times \mu} \ar[d]
& G^{\#} \times X \ar[d] \\
\tilde{X}^{\#} \ar[r]^{\eta^{\#}} & X^{\#}
}\]
where the vertical arrows are $G^0$-torsors.
As $\eta^{\#}$ is finite and birational, 
the same holds for $\id \times \mu$,
and hence for $\mu$; in particular,
$Y$ is a variety. 

Next, the $G^0$-normalization
$\varphi : X' \to X$ yields 
a $G^{\#}$-equivariant morphism
\[ G^{\#} \times^{G^0} \varphi :  
G^{\#} \times^{G^0} X' \longrightarrow
G^{\#} \times^{G^0} X = X^{\#} \]
which is finite and birational 
by the above argument. By the universal
property of the normalization, 
$\eta^{\#}$ lifts to a unique morphism
$\gamma : \tilde{X}^{\#} \to 
G^{\#} \times^{G^0} X'$ which is finite
and birational, and hence 
$G^{\#}$-equivariant. Moreover,
$\gamma$ is a morphism of varieties
over $G^{\#}/G^0$, and hence
restricts to a finite birational
$G^0$-morphism $\delta : Y \to X'$.
Since $X'$ is $G^0$-normal,
$\delta$ is an isomorphism. 
So $Y$ is $G^0$-normal, and hence 
$G$-normal by Remark \ref{rem:Gnormal}.
\end{proof}

As a direct consequence of Lemma 
\ref{lem:induced}, we obtain:

\begin{corollary}\label{cor:induced}
The following conditions are equivalent 
for a $G$-variety $X$:

\begin{enumerate}

\item[{\rm (i)}]
$X$ is $G$-normal.

\item[{\rm (ii)}]
$X$ is $G^0$-normal.

\item[{\rm (iii)}]
$X^{\#}$ is normal. 

\end{enumerate}

\end{corollary}

\begin{proposition}\label{prop:field}
Let $X$ be a $G$-variety, and $k'/k$ 
a field extension. 

\begin{enumerate}

\item[{\rm (i)}] If $X_{k'}$ is 
$G_{k'}$-normal, then $X$ is $G$-normal.

\item[{\rm (ii)}] If $X$ is $G$-normal 
and $k'$ is separable over $k$, 
then $X_{k'}$ is $G_{k'}$-normal.

\end{enumerate}

\end{proposition}

\begin{proof}
(i) Consider a finite birational morphism 
of $G$-varieties $f : Y \to X$. Then 
the base change
$f_{k'} : Y_{k'} \to X_{k'}$ is a finite
birational morphism of $G_{k'}$-varieties, 
and hence an isomorphism. By descent, 
$f$ is an isomorphism as well.

(ii) By Corollary \ref{cor:induced},
we may assume that $G$ is connected.
Then the assertion follows from this 
corollary, since the formation
of $X^{\#}$ commutes with field extensions,
and normality is preserved under separable
field extensions (see 
\cite[${\rm IV}_{\rm 2}$.6.7.4]{EGA}
for the latter assertion).
\end{proof}

\begin{remark}\label{rem:nonperfect}
It is well known that normality may not 
be preserved 
under a non-trivial purely inseparable field 
extension $k'/k$ (see e.g. 
\cite[${\rm IV}_{\rm 2}$.6.7.5]{EGA}). 
This also holds for $G$-normality where 
$G = \alpha_p$, as shown by the following
example: Choose $a \in k$ such that 
$a^{1/p} \in k' \setminus k$.
Let $G^{\#}$ be the affine plane curve 
with equation $y^p = x + a x^p$. Then 
$G^{\#}$ is a smooth connected algebraic group 
via pointwise addition of coordinates,
and its Frobenius kernel (the zero subscheme 
of $(x^p,y^p)$) is isomorphic to $G$.
Moreover, one may check that the normal 
projective completion of $G^{\#}$ is the
projective plane curve $X$ with homogeneous 
equation $y^p = xz^{p-1} + ax^p$. 
The $G^{\#}$-action on itself by translation 
extends uniquely to an action on $X$. 
In particular, $X$ is a normal $G$-curve,
and hence is $G$-normal. But $X_{k'}$ 
is non-normal, since it has homogeneous equation
$(y  - a^{1/p}x)^p = x z^{p-1}$. As $G^{\#}_{k'}$ 
is smooth, its action on $X_{k'}$ lifts  
to an action on the normalization of this 
curve, and hence to a $G_{k'}$-action. It follows 
that $X_{k'}$ is not $G_{k'}$-normal.
\end{remark}

\begin{remark}\label{rem:Gintegral}
Consider an \emph{affine} $G$-variety $X$; 
then the normalization $\tilde{X}$ and
the $G$-normalization $X'$ are affine
as well, and we have inclusions of rings
\[ R = \cO(X) \subset R' = \cO(X') \subset 
\tilde{R} = \cO(\tilde{X}) \subset 
K = k(X) = k(X') = k(\tilde{X}). \]
Also, recall that $K$ is a $G$-module, 
and $R$, $R'$ are submodules (but 
$\tilde{R}$ is generally not a submodule).

We say that $f \in K$ is 
\emph{$G$-integral} (over $R$)
if the $G$-submodule of $K$ generated 
by $f$ is contained in $\tilde{R}$. 
We now claim that $R'$ is the subset 
of $K$ consisting of $G$-integral elements.

Indeed, since the direct sum and
tensor product of any two $G$-modules are $G$-modules
and the sum and product of integral elements are 
integral, we see that the $G$-integral elements form 
a subalgebra $S \subset K$. Clearly, we have
$R' \subset S \subset \tilde{R}$, 
and hence $S$ is a finite 
$R'$-module. In particular, $\Spec(S)$ is 
a $G$-variety equipped with a finite birational 
equivariant morphism to $X'$. Thus, 
$S = R'$, proving the claim. 

As a direct consequence of this claim, 
the formation of $R'$ commutes with 
localization by $G$-invariants of $R$. 

Next, consider a $G$-variety $X$ admitting
a covering by open affine $G$-stable subsets. 
Then the $G$-normalizations of these subsets may 
be glued to a $G$-variety, which is readily seen
to be the $G$-normalization. This provides an 
algebraic construction of the $G$-normalization.
\end{remark}

We now obtain an equivariant version of 
Serre's criterion for normality 
(see \cite[${\rm IV}_{\rm 2}$.5.8.6]{EGA}).
The latter can be stated as follows in our
setting: a variety $X$ is normal if and only
if it satisfies $(S_2)$ and the ideal sheaf 
of every closed subvariety is invertible 
in codimension $1$.

\begin{theorem}\label{thm:serre}
Let $X$ be a $G$-variety. Then $X$ is $G$-normal
if and only if it satisfies $(S_2)$ and 
the ideal sheaf $\cI_Z$ is invertible 
in codimension $1$ for any closed $G$-stable 
subscheme $Z \subsetneq X$.
\end{theorem}

\begin{proof}
We use the construction before Lemma
\ref{lem:induced}. By Corollary 
\ref{cor:induced}, it suffices to show 
the following two equivalences:

\begin{enumerate}
\item[{\rm (i)}] 
$X^{\#}$ satisfies $(S_2)$ if and only if
$X$ satisfies $(S_2)$.

\item[{\rm (ii)}]
$X^{\#}$ satisfies $(R_1)$ if and only if
$\cI_Z$ is invertible in codimension $1$ 
for any closed $G$-stable subscheme 
$Z \subsetneq X$.

\end{enumerate}

(i) This follows from the fact that 
$\psi : X^{\#} \to G^{\#}/G$ 
is a faithfully flat morphism to a smooth
variety, with fibers being obtained from
$X$ via base change by field extensions
(use \cite[${\rm IV}_{\rm 2}$.6.6.1]{EGA}).

For (ii), assume that $X^{\#}$
satisfies $(R_1)$. Let $Z \subsetneq X$
be a closed $G$-stable subscheme; then
$Z^{\#} = G^{\#} \times^{G^0} Z$
is a closed subscheme of $X^{\#}$, and
hence $\cI_{Z^{\#}}$ is invertible
in codimension $1$. Since the quotient
$G^{\#} \times X \to X^{\#}$ is a
$G^0$-torsor, it follows that
$\cI_{G^{\#} \times Z} 
\subset \cO_{G^{\#} \times Z}$
is invertible in codimension $1$ as well.
This yields the desired assertion.

Conversely, assume that $\cI_Z$ 
is invertible in codimension $1$ 
for any closed $G$-stable subscheme 
$Z \subsetneq X$.
Recall that $X^{\#}$ is equipped with
an action of the smooth algebraic group 
$H = G^{\#} \times \pi_0(G)$; thus, the
regular locus $X^{\#}_{\reg}$ is $H$-stable.
It follows that the singular locus 
$X^{\#}_{\sing}$ (equipped with its reduced 
subscheme structure) is $H$-stable as well
(indeed, the formation of $X^{\#}_{\sing}$
commutes with separable field extensions,
and hence we may assume $k$ separably
closed. Then $X^{\#}_{\sing}$ is stable
under $H(k)$, and hence under its
schematic closure $H$). So 
$X^{\#}_{\sing} = H \times^G Z$ for
a unique closed $G$-stable subscheme
$Z \subsetneq X$. By using the $G$-torsor
$H \times X \to X^{\#}$ as above,
it follows that $\cI_{X^{\#}_{\sing}}$
is invertible in codimension $1$.
This forces 
$\codim_{X^{\#}}(X^{\#}_{\sing}) \geq 2$.
\end{proof}

\begin{example}\label{ex:Gnormal}
Assume that $k$ is algebraically closed.
Consider the zero subscheme 
$X \subset \bA^{n+1}$
of $y^{p^m} -f(x_1,\ldots,x_n)$, where 
$m,n$ are positive integers, $x_1,\ldots, x_n,y$
denote the coordinates on $\bA^{n+1}$, and 
$f \in k[T_1,\ldots,T_n]$. 
The group scheme $\alpha_{p^m}$
acts freely on $\bA^{n+1}$ via 
$g \cdot (x_1,\ldots,x_n,y) = 
(x_1,\ldots,x_n, g + y)$ 
and this action stabilizes $X$. 
The quotient is the morphism
\[ \bA^{n+1} \longrightarrow \bA^{n+1},
\quad (x_1,\ldots,x_n,y) \longmapsto 
(x_1,\ldots,x_n,y^{p^m}). \]
Its restriction to $X$ is identified
with the projection
$(x_1,\ldots,x_n) : X \to \bA^n$.
So $X$ is $\alpha_{p^m}$-normal
by Corollary \ref{cor:unibranch}.
Also, $X$ is generally singular 
in codimension $1$, e.g., when 
$f$ is divisible by the square of 
a non-constant polynomial; then
the singular locus is not stable under
$\alpha_{p^m}$.

Next, let $\mu_{p^m}$ act on $\bA^{n+1}$
via 
$t \cdot (x_1,\ldots,x_n,y) = 
(x_1,\ldots,x_n, t y)$. Then this action
stabilizes $X$, and the quotient is
as above. One may check by using
Theorem \ref{thm:serre} that 
$X$ is not $\mu_{p^m}$-normal when
$f$ is divisible by the square of 
a non-constant polynomial.
\end{example}

Next, we obtain an equivariant version 
of a classical normality criterion 
for curves:

\begin{corollary}\label{cor:Gnormal}
The following conditions are equivalent
for a $G$-curve $X$:

\begin{enumerate}
 
\item[{\rm (i)}] $X$ is $G$-normal.

\item[{\rm (ii)}] For any closed point $x \in X$, 
the ideal $\cI_{G \cdot x}$ is invertible.

\item[{\rm (iii)}] For any closed $G$-stable 
subscheme $Z \subsetneq X$, the ideal $\cI_Z$ 
is invertible.

\end{enumerate}

\end{corollary}

\begin{proof}
(i)$\Rightarrow$(ii) Let $x  \in X$ be a closed point.
Set $Z = G \cdot x$ and consider the blow-up
$f : \Bl_Z(X) \to X$. Then $\Bl_Z(X)$ is a $G$-curve 
and $f$ is equivariant. Thus, $f$ is an isomorphism. 
By the universal property of the blow-up, this means 
that $\cI_Z$ is invertible.

(ii)$\Rightarrow$(iii) Let $Z \subsetneq X$ be a closed 
$G$-stable subscheme. We show that $\cI_Z$ is invertible 
by induction on the length
$\ell(Z) = \dim H^0(X, \cO_X/\cI_Z)$. 
We may assume that $\ell(Z) \geq 1$,
and hence choose a closed point $x \in Z$. Then 
$G \cdot x \subset Z$; thus, $\cI = \cI_Z$ is contained 
in $\cI_{G \cdot x} = \cJ$ and the latter is invertible. 
So $\cJ^{-1} \cI$ is a $G$-stable sheaf of ideals of 
$\cO_X$. Denoting by $W \subset X$ the corresponding 
closed $G$-stable subscheme, we have 
\[ 
\cO_X/\cI_W = \cO_X/\cJ^{-1} \cI \simeq
\cJ/(\cJ^{-1}\cI \otimes_{\cO_X} \cJ) \simeq
\cJ/\cI. \]
As a consequence,
$\ell(W) = \dim H^0(X,\cJ/\cI) < \ell(Z)$.
By the induction assumption, $\cI_W$ 
is invertible, and hence so is $\cI_Z$.

(iii)$\Rightarrow$(i) This follows readily 
from Theorem \ref{thm:serre}.
\end{proof}

\begin{remark}\label{rem:Gnormal}
Given a smooth closed point $x$ of
a $G$-curve $X$, 
the orbit $G \cdot x$ is contained in 
the smooth locus of $X$, and hence 
the ideal $\cI_{G \cdot x}$ is invertible. 
So to check the $G$-normality of $X$, 
it suffices to show that $\cI_{G \cdot x}$ 
is invertible for any non-smooth point $x$. 

Likewise, the $G$-normalization of $X$ 
is obtained by iterating the process of 
blowing up the $G$-orbits of non-smooth points. 
\end{remark}

\begin{corollary}\label{cor:prime}
The following conditions are equivalent for 
a finite group scheme $G$ of order $p$
and a $G$-curve $X$:

\begin{enumerate}

\item[{\rm (i)}] $X$ is $G$-normal.

\item[{\rm (ii)}] $X/G$ is normal and 
$X$ is normal at every $G$-stable point. 
\end{enumerate}

\end{corollary}

\begin{proof}
(i)$\Rightarrow$(ii) The normality of $X/G$ 
follows from Corollaries \ref{cor:unibranch}
and \ref{cor:induced}.
If $x \in X$ is $G$-stable, then the ideal 
$\cI_x$ is invertible by Corollary
\ref{cor:Gnormal},
hence $X$ is normal at $x$.

(ii)$\Rightarrow$(i) Let $x \in X$ be a closed point.
If $x$ is $G$-stable, then $\cI_{G \cdot x} = \cI_x$
is invertible. Otherwise, we claim that 
$\Stab_G(x)$ is trivial. To check this, view $x$ 
as a $\kappa(x)$-point of the $G_\kappa(x)$-variety
$X_{\kappa(x)}$; then $x$ is not 
$G_{\kappa(x)}$-stable, since its orbit 
$G_{\kappa(x)} \cdot x$ has schematic image
$G \cdot x$ under the projection
$X_{\kappa(x)} \to X$. As 
$G_{\kappa(x)} \cdot x \simeq 
G_{\kappa(x)}/\Stab_G(x)$, it follows that
$\Stab_G(x)$ is strictly contained in 
$G_{\kappa(x)}$. This implies our claim, since
$G$ has order $p$. 

By the claim, the quotient
$q : X \to X/G$ is a $G$-torsor at $x$, and hence
$\cI_{G \cdot x} = \cI_{q(x)} \cO_X$. Since 
the curve $X/G$ is normal, $\cI_{q(x)}$ is
invertible; therefore, so is $\cI_{G \cdot x}$. 
By Corollary \ref{cor:Gnormal} again, 
it follows that $X$ is $G$-normal.
\end{proof}

Next, assume that $G$ is a subgroup scheme of 
a smooth connected algebraic group $G^{\#}$. 
Then the quotient $G^{\#} \times^G X$
exists for any $G$-curve $X$, 
since $G^{\#}$ and $X$ are quasi-projective.

\begin{proposition}\label{prop:regular}
With the above notation and assumptions,
$X$ is $G$-normal if and only if 
$G^{\#} \times^G X$ is regular.
\end{proposition}

\begin{proof}
Note that there is a canonical morphism
$X^{\#} = G^{\#} \times^{G^0} X
\to G^{\#} \times^G X$,
which is a torsor under the finite \'etale
group scheme $G/G^0 = \pi_0(G)$.  

If $G^{\#} \times^G X$ is regular, 
then so is $X^{\#}$ by 
\cite[${\rm IV}_{\rm 2}$.6.6.1]{EGA}. 
So $X$ is $G$-normal in view of Corollary 
\ref{cor:induced}.

The converse is obtained by arguing 
as in the end of the proof of 
Theorem \ref{thm:serre}: 
the singular locus of $X^{\#}$ satisfies
$X^{\#}_{\sing} = G^{\#} \times^G Z$
for some closed $G$-stable subscheme 
$Z \subsetneq X$. Thus, 
$\codim_Z(X) = 
\codim_{X^{\#}}(X^{\#}_{\sing})$.
Since $X^{\#}$ is normal (by Corollary
\ref{cor:induced} again), this yields
$\codim_X(Z) \geq 2$ and hence 
$Z = \emptyset = X^{\#}_{\sing}$.
So $X^{\#}$ is regular, and hence 
$G^{\#} \times^G X$ is regular as well
(see \cite[${\rm IV}_{\rm 2}$.6.6.1]{EGA} 
again).
\end{proof}

\begin{corollary}\label{cor:lci}
Every $G$-normal curve is a local 
complete intersection. 
\end{corollary}

\begin{proof}
Given a $G$-normal curve $X$,
consider the natural morphism 
$\psi : X^{\#} = G^{\#} \times^G X 
\to G^{\#}/G$
with fiber $X$ at the base point.
Note that $X^{\#}$ is regular, 
$G^{\#}/G$ is smooth, and $\psi$
is faithfully flat (e.g., by 
$G^{\#}$-equivariance). So the assertion
follows from 
\cite[${\rm IV}_{\rm 4}$.19.3.2]{EGA}.
\end{proof}

\begin{remark}\label{rem:stack}
Assume that $k$ is perfect. Then a $G$-curve
$X$ is $G$-normal if and only if the
quotient stack $[X/G]$ is smooth. This follows
from the above proposition in view of the
isomorphism of stacks
$[X/G] \simeq [G^{\#} \times^G X/G^{\#}]$,
which in turn is a direct consequence of
the definitions of such stacks. 
\end{remark}

\section{Generically free actions on curves}
\label{sec:gen}

Throughout this section, we denote by $G$
an infinitesimal group scheme of order 
$p^n$, with Lie algebra $\fg$.

\begin{proposition}\label{prop:inf}
Let $X$ be a curve equipped with a rational 
action of $G$; let $K = k(X)$ and $L = K^G$. 
Then there exists a unique integer 
$m = m(G) \geq 0$ such that $L = k K^{p^m}$. 
Moreover, $m \leq n$ and equality holds if and 
only if the rational $G$-action on $X$ is 
generically free.
\end{proposition}

\begin{proof}
Recall that the extension $K/L$ 
is finite and purely inseparable
(Lemma \ref{lem:field}). Thus,
$[K:L] = p^m$ for some $m \geq 0$, and 
$k K^{p^m} \subset L$. But since $K$ is 
a function field in one variable, we have
$[K : k K^{p^m}] = p^m$ (indeed, by 
using the tower of field extensions
$K = K_0 \supset k K^p = K_1 \supset
\ldots \supset kK^{p^m} = K_m$
where $K_{i+1} = k K_i^p$ for 
$i = 0,\ldots,m-1$,  
it suffices to show that $[K : kK^p] = p$.
But this follows from the fact that 
every $x \in K \setminus kK^p$ forms
a $p$-basis of the function field in one 
variable $K$ over $k$, see e.g.
\cite[${\rm IV}_{\rm 1}$.2.1.4]{EGA}).
So $L = k K^{p^m}$. The final assertion
follows readily from Corollary 
\ref{cor:degree}.
\end{proof}

\begin{remark}\label{rem:inf}
Let $X$ be a generically free $G$-curve
with quotient $q : X \to Y = X/G$.
Then $X^{(p^n)}$ is a curve with function 
field $k K^{p^n}$, and the morphism 
$F^n_X : X \to X^{(p^n)}$ is finite,
surjective and $G$-invariant. 
Thus, $F^n_X$ factors uniquely as
\[ X \stackrel{q}{\longrightarrow}
Y \stackrel{r}{\longrightarrow} X^{(p^n)},
\]
where $r$ is finite surjective as well;
also, $r$ is birational by
Proposition \ref{prop:inf}.

If $X$ is $G$-normal, then $Y$ is normal 
(Lemma \ref{lem:quotient}) and hence is 
isomorphic to the normalization of 
$X^{(p^n)}$. As a consequence, $g(Y) = g(X)$
where $g$ denotes the geometric genus.
\end{remark}

\begin{lemma}\label{lem:gen}
The following conditions are equivalent
for a curve $X$ equipped with a faithful
action of $G$:

\begin{enumerate}

\item[{\rm (i)}]
The $G$-action is generically free.

\item[{\rm (ii)}] We have $\dim(\fg) = 1$. 

\item[{\rm (iii)}]
We have an isomorphism of algebras
$\cO(G) \simeq k[T]/(T^{p^n})$.

\end{enumerate}

Under these conditions, $G$ is either unipotent
or a form of $\mu_{p^n}$.
\end{lemma}

\begin{proof}
(i)$\Leftrightarrow$(ii) We may assume 
$k$ algebraically closed. Then (i) is equivalent
to the existence of a smooth point $x \in X(k)$ 
such that $\Stab_G(x)$ is trivial. This is in turn
equivalent to $\Lie(\Stab_G(x)) = 0$, since $G$
is infinitesimal. But $\Lie(\Stab_G(x))$ is the
kernel of the natural map $\fg \to T_x X$ 
(the differential at $e$ of the morphism 
$G \to X$, $g \mapsto g \cdot x$), see e.g. 
\cite[III.2.2.6]{DG}. So $\Lie(\Stab_G(x)) = 0$
implies that $\dim(\fg) = 1$. 

Conversely, assume that $\dim(\fg) = 1$.
Choose a smooth point $x \in X(k)$ which is not
fixed by $\fg$ (identified with the Frobenius
kernel $G_1$). Then $\Stab_G(X)_1$ is trivial,
and hence so is $\Stab_G(x)$.

(ii)$\Leftrightarrow$(iii) This follows from 
the fact that $\cO(G)$ is a local $k$-algebra 
of dimension $p^n$ and residue field $k$.

To show the final assertion, we may again 
assume $k$ algebraically closed. If $G$
contains no copy of $\alpha_p$, then 
$G$ is diagonalizable by \cite[IV.3.3.7]{DG}.
Thus, $G \simeq \prod_{i = 1}^r \mu_{p^{m_i}}$
in view of the structure of diagonalizable 
group schemes (see e.g. \cite[IV.1.1.2]{DG}). 
Using (ii), it follows that $r = 1$.

So we may assume that $\alpha_p \subset G$.
By (ii) again, we then have 
$G_1 = \alpha_p$. The relative Frobenius
$F_G$ yields an exact sequence
\[ 1 \longrightarrow G_1 \longrightarrow G
\stackrel{f}{\longrightarrow} F(G) 
\longrightarrow 1,
\]
where $\cO(F(G)) = \cO(G)^p = k[T]/(T^{p^{n-1}})$
in view of (iii). Arguing by induction on $n$,
we may assume that $F(G)$ is either unipotent
or isomorphic to $\mu_{p^{n-1}}$. In the
former case, $G$ is unipotent. In the latter 
case, $G$ is trigonalizable (see
\cite[IV.2.3.1]{DG}) and hence
$G \simeq \alpha_p \rtimes \mu_{p^{n-1}}$
by loc.~cit., IV.2.3.5. But then
$\dim(\fg) = 2$, a contradiction.
\end{proof}

\begin{proposition}\label{prop:tangent}
Let $X$ a generically free $G$-normal curve.

\begin{enumerate}
\item[{\rm (i)}] 
The quotient of $\Omega^1_{X/k}$ by its 
torsion subsheaf is invertible.

\item[{\rm (ii)}] 
The tangent sheaf $\cT_X$ is invertible
and effective.

\end{enumerate}

\end{proposition}

\begin{proof}
(i) Choose a non-zero element of $\fg$, which 
yields a morphism of $\cO_X$-modules
$D : \Omega^1_{X/k} \to \cO_X$. Then $D$ is 
an isomorphism at the generic point, and hence 
its kernel is the torsion subsheaf. Also,
$\Omega^1_{X/k}$ is equipped with a 
$G$-linearization, and $D$ is $G$-equivariant 
(as the adjoint representation of $G$ 
in $\fg$ is trivial by the final assertion
of Lemma \ref{lem:gen}).
Thus, the image of $D$ is a $G$-stable 
ideal. As every such ideal is invertible 
(Corollary \ref{cor:Gnormal}), 
this yields the assertion.

(ii) The sheaf $\cT_X$ is invertible by (i),
and has non-zero global sections. 
\end{proof}

\begin{proposition}\label{prop:local}
Let $X$ a generically free $G$-normal curve 
with quotient $Y$. Let $x \in X$ be 
a closed $G$-fixed point with 
image $y \in Y$. If $k$ is perfect, 
then there exists an isomorphism of 
$\cO_{Y,y}$-algebras
\[ \cO_{X,x} \simeq 
\cO_{Y,y}[T]/(T^{p^n} - w), \]
where $w$ generates the maximal ideal of 
$\cO_{Y,y}$. 
\end{proposition}

\begin{proof}
By Corollaries \ref{cor:unibranch}
and \ref{cor:induced},
$Y$ is smooth at $y$;
also, $X$ is smooth at $x$ in view of
Corollary \ref{cor:Gnormal}.
Since $G$ is infinitesimal, the fiber 
$X_y$ has a unique point $x$ and
the extension $\kappa(x)/\kappa(y)$ is 
purely inseparable. But $\kappa(y)$ is 
perfect (since so is $k$), and hence 
$\kappa(x) = \kappa(y)$. Thus, 
$\cO_{X,x}$ and $\cO_{Y,y}$ are discrete 
valuation rings with the same residue field
$k' = \kappa(x)$; 
moreover, $\cO_{X,x}$ is a finite free 
module over $\cO_{Y,y}$. 
The rank of this module is $p^n$, as 
$[ k(X) : k(Y) ] = p^n$ by Proposition
\ref{prop:inf}. Thus, $X_y$ is a finite
$k'$-subscheme of length $p^n$ of 
$\cO_{X,x}$.  As a consequence,
$\cO(X_y) = \cO_{X,x}/(z^{p^n})$, where 
$z \in \cO_{X,x}$ generates the maximal ideal 
$\fm_x$. On the other hand, 
$\cO(X_y) = \cO_{X,x}/(w)$, where 
$w \in \cO_{Y,y}$ generates $\fm_y$. 
So $z^{p^n} = u w$ where $u$ is a unit 
in $\cO_{X,x}$.
By Proposition \ref{prop:inf}
and the normality of $Y$, we have 
$z^{p^n} \in \cO_{Y,y}$. Thus, $u$ 
is a unit in $\cO_{Y,y}$. Replacing $w$ with  
$uw$, we may thus assume that $z^{p^n} = w$. 
Consider the homomorphism of 
$\cO_{Y,y}$-algebras 
\[ \cO_{Y,y}[T]/(T^{p^n} -w) \longrightarrow
\cO_{X,x}, \quad T \longmapsto z. \]
This is surjective by Nakayama's lemma, and hence
an isomorphism since both sides are free 
$\cO_{Y,y}$-modules of rank $p^n$.
\end{proof}

Finally, we use the above results and methods 
to settle an issue in the classification of
maximal connected algebraic groups of 
birational automorphisms of surfaces. 
This classification (up to conjugation by 
birational automorphisms) was recently obtained 
by Fong for smooth projective surfaces over 
an algebraically closed field, see \cite{Fong}.
One case was left unsettled in positive 
characteristics, see Proposition 3.25 
and Remark 3.26 in loc.~cit. 
This case can be handled as follows:

Assume that $k$ is algebraically closed
and let $S$ be a smooth projective surface
equipped with a faithful action of an
elliptic curve $E$. By loc.~cit., there is 
an $E$-equivariant isomorphism
$S \simeq E \times^G X$, where $G \subset E$
is a finite subgroup scheme,
and $X \subset S$ is a closed $G$-stable 
subscheme of pure dimension $1$
(a priori, $X$ is not necessarily reduced).
This yields an $E$-equivariant morphism
$\psi : S \to E/G$ with fiber $X$
at the origin of the elliptic curve $E/G$.
Using the Stein factorization, 
we may assume that 
$\psi_*(\cO_S) = \cO_{E/G}$; then 
$X$ is connected.

\begin{proposition}\label{prop:reduced}
With the above notation and assumptions,
$X$ is a $G$-normal curve.
\end{proposition}

\begin{proof}
We begin with some observations. 
First, $X$ is Cohen-Macaulay, as 
$\psi$ is flat. 

Also, the categorical quotient
$X \to X/G = Y$ exists, and the natural 
morphism $S = E \times^G X \to Y$ is
the categorical quotient by $E$. It follows
that $Y$ is a smooth projective curve.

Finally, $G$ acts faithfully on $X$, since 
$E$ is a commutative algebraic group and 
acts faithfully on $S$.

We now reduce to the case where $G$ is
infinitesimal. For this, we consider
the natural morphism
$f : E \times^{G^0} X \to E \times^G X = S$,
which is a torsor under the finite \'etale
group scheme $\pi_0(G)$. Thus, 
$E \times^{G^0} X$ is smooth, projective 
and of pure dimension $2$. It is also 
connected (since so are $E$ and $X$) 
and equipped with a faithful action of $E$.
This yields the desired reduction.

The scheme $X$ is irreducible, as the 
quotient morphism $E \times X \to S$
is a homemorphism. Also, since $G$ is 
a subgroup scheme of $E$,
its Lie algebra has dimension $1$ and hence
$\cO(G) \simeq k[T]/(T^{p^n})$ as algebras 
(Lemma \ref{lem:gen}). As a consequence,
the Frobenius kernel $G_1$ has order $p$.

We now show that the $G$-action on $X$
is generically free. Since the $G_1$-action
is faithful, there exists $x \in X(k)$
such that $\Stab_{G_1}(x)$ is trivial
(Lemma \ref{lem:gen} again).
Then $\Stab_G(x)$ has a trivial Frobenius
kernel, and hence is trivial as well.

Next, we describe the local structure of 
$X$ at $x \in X_{\fr}(k)$, by adapting 
the argument of Proposition 
\ref{prop:local}. Let $y = q(x) \in Y(k)$. 
Then the quotient $q : X \to Y$ is a $G$-torsor
at $y$, and hence is finite free of
rank $p^n$ at that point. The fiber $X_y$
satisfies 
$\cO(X_y) \simeq k[T]/(T^{p^n})$. Choose
a lift $z \in \cO_{X,x}$ of the class of
$T$ in $\cO(X_y)$. Then 
$z^{p^n} \in \cO_{X,x}^G$, since 
the relative Frobenius morphism 
$F^n_X$ is $G$-invariant. Thus,
$z^{p^n} = w$ where $w \in \cO_{Y,y}$.
As $z$ lifts $T$, we get $w(y) = 0$,
i.e., $w \in \fm_y$. So the homomorphism
of $\cO_{Y,y}$-algebras
\[ f : \cO_{Y,y}[T]/(T^{p^n} - w) 
\longrightarrow \cO_{X,x}, \quad 
T \longmapsto z \]
induces an isomorphism modulo $\fm_y$.
By Nakayama's lemma, it follows that
$f$ is an isomorphism.

Now consider the case where 
$w \notin k(Y)^{p^n}$.
Then $\cO_{X,x}$ is reduced. Since
$X$ is Cohen-Macaulay, it is reduced
as well. 
So $X$ is a $G$-curve; it is $G$-normal
by Proposition \ref{prop:regular}.

It remains to treat the case where 
$w \in k(Y)^{p^n}$. Since $Y$ is normal,
we then have $w = t^{p^n}$ where
$t \in \cO_{Y,y}$. So 
$\cO_{X,x} \simeq \cO_{Y,y}[U]/(U^{p^n})$
as an $\cO_{Y,y}$-algebra. In geometric 
terms, there exists a dense open subset
$V \subset Y$ and a section 
$\sigma : V \to X_{\fr}$
of $q$ above $V$. The schematic closure 
of $\sigma(V)$ in $X$ is a projective 
curve with smooth projective model $Y$.
Thus, $\sigma$ extends to a section 
$\tau : Y \to X$ of $q$.
This yields a morphism
\[ \varphi : E \times Y 
\longrightarrow S, \quad
(e,y) \longmapsto e \cdot \tau(y) \]
which is $E$-equivariant, where 
$E$ acts on $E \times Y$ by translation
on itself. Moreover, $\varphi$ restricts 
to an isomorphism
$G \times V \to q^{-1}(V)$,
since $q$ induces a $G$-torsor
$q^{-1}(V) \to V$ with section $\sigma$.
Thus, $\varphi$ restricts to an 
isomorphism
$E \times^G (G \times V) \to
E\times^G q^{-1}(V)$,
i.e., an isomorphism of $E \times V$
onto an open subset of $S$. 
In particular, $\varphi$ is birational.
Also, the triangle
\[ \xymatrix{
E \times Y \ar[rr]^-{\varphi} 
\ar[dr]^{\gamma} 
& & S \ar[dl]_{\psi} \\
& E/G & \\ 
}\]
commutes, where $\gamma$ is the
composite of the projection
$E \times Y \to E$ with the quotient
map $E \to E/G$. Moreover,  
$\gamma$ and $\psi$ have
irreducible fibers. Thus, $\varphi$
is finite, and hence an isomorphism.
Since $\psi_*(\cO_S) = \cO_{E/G}$,
it follows that $G$ is trivial
and $X = Y$.
\end{proof}

\begin{remark}\label{rem:totaro}
Proposition \ref{prop:reduced}
can be extended to a perfect ground
field $k$. But it fails over any
imperfect field in view of the existence 
of non-trivial pseudo-abelian surfaces, 
see \cite[Sec.~6]{Totaro}
and \cite[Rem.~6.4 (ii)]{Br17a}.
\end{remark}

\begin{remark}\label{rem:reduced}
Conversely, given a $G$-normal curve $X$ 
over an algebraically closed field and 
an embedding of $G$ into an elliptic curve $E$, 
we obtain a smooth projective surface 
$S = E \times^G X$
by Proposition \ref{prop:regular}. 
Moreover, the morphism $f : S \to X/G = Y$ 
(the categorical quotient by $E$)
is an elliptic fibration with general fiber 
$E$. The multiple fibers of $f$ are exactly 
the fibers $S_y$, where $y = f(x)$ and 
$x$ has a non-trivial stabilizer; the 
multiplicity of $S_y$ is the order of
$\Stab_G(x)$. Also, one may easily check 
that $R^1 f_*(\cO_S) = \cO_Y$.

If $X$ is rational (or equivalently, 
$Y$ is rational; see Remark \ref{rem:inf}), 
then the natural morphism $S \to E/G$ is
the Albanese morphism, since its
fibers at closed points are the 
translates of $X$. 
Moreover, these fibers are exactly the
rational curves on $S$. If in addition
$X$ is singular, then it follows that
$S$ is minimal. Using the canonical bundle 
formula for the elliptic fibration $f$
(see \cite{BMII}), one may check that 
the canonical divisor $K_S$ is numerically 
equivalent to a positive multiple of $E$
when $p >3$ and $f$ has at least $3$
multiple fibers. In particular, $S$ has 
Kodaira dimension $1$ under these
assumptions. But $S$ may well have 
Kodaira dimension $0$, for example when
$p \leq 3$ and $S$ is a quasi-hyperelliptic 
surface (see \cite{BMIII}).
\end{remark}

\begin{example}\label{ex:Gnormalbis}
Assume that $k$ is algebraically closed
and consider a projective variant of
Example \ref{ex:Gnormal}: let $X$ be 
the projective plane curve with equation
$z^{p^n} -f(x,y)$, where $f$ is 
a homogeneous polynomial of degree $p^n$
and is not a $p$th power. Assume 
in addition that $p^n \geq 3$; then 
$X$ is singular. 

Consider the action of $\alpha_{p^n}$ 
on $\bP^2$ via 
$g \cdot (x,y,z) = (x,y,z + g(ax + by))$,
where $a,b \in k$. Then $X$ is stable
under this action, and the quotient morphism
is the projection $[x:y] : X \to \bP^1$.
So $X$ is rational by Remark 
\ref{rem:inf} or a direct argument.
Using Corollary \ref{cor:Gnormal}
and the subsequent remark, one may check
that $X$ is $G$-normal for general $a,b$.

Next, let $\mu_{p^n}$ act on $\bP^2$ via 
$t \cdot (x,y,z) = (x,y,tz)$. Then again,
$X$ is stable under this action, with
the same quotient. Moreover, $X$ is
$\mu_{p^n}$-normal if and only if 
$f$ is square-free.

Every elliptic curve $E$ contains copies
of $\mu_{p^n}$ (if $E$ is ordinary) or 
$\alpha_{p^n}$ (if $E$ is supersingular)
for all $n \geq 1$. Denote by 
$G \subset E$ the resulting infinitesimal 
subgroup scheme, and let $X$ be a singular
$G$-normal curve obtained by the above 
construction. 
In view of Remark \ref{rem:reduced},
this yields a smooth projective
minimal surface $S = E \times^G X$,
which is uniruled but not ruled.
By the canonical bundle formula for
the elliptic fibration $S \to \bP^1$,
the Kodaira dimension of $S$ is $1$ 
if $p^n \geq 4$.

\end{example}

\section*{Acknowledgements} 
I thank Pascal Fong, J\'er\'emy Gu\'er\'e, 
Daniel Litt,
Matilde Maccan, Benjamin Morley, 
Immanuel van Santen, Ronan Terpereau 
and Susanna Zimmermann for very helpful 
comments and discussions. Also, many 
thanks to the anonymous referee for 
a careful reading and valuable suggestions.

\end{document}